\documentclass{amsart}
\usepackage{latexsym}
\usepackage{amsmath,amssymb,amsfonts,amsthm,eufrak}
\usepackage{mathrsfs}
\usepackage{upref}
\usepackage{amscd}
\usepackage{eucal}
\usepackage[all]{xy}
\usepackage[dvips]{graphicx}

\title[On the slope of fibrations]{On the slope of relatively minimal fibrations on rational complex surfaces}
\author{Claudia R. Alc\'antara, Abel Castorena, Alexis G. Zamora}
\date{}
\thanks{The first author was partially supported by CONACyT Grant 058486. The second author was partially supported by CONACyT Grant 48668-F and
PAPIIT Grant
IN100909-2. The third author was partially supported by CONACyT Grants 25811.}
\subjclass[2000]{Primary 14D06, 14J26}
\keywords{rational surface, fibration.}

\begin{document}

\newtheorem{teo}{Theorem}[section]
\newtheorem{lem}[teo]{Lemma}
\newtheorem{defin}[teo]{Definition}
\newtheorem{prop}[teo]{Proposition}
\newtheorem{cor}[teo]{Corollary}
\newtheorem{remark}[teo]{Remark}

\newcommand{\Hc}{\mathscr{H}}
\newcommand{\Oc}{\mathcal{O}}
\newcommand{\Tc}{\mathscr{T}}
\newcommand{\CP}{\mathbb{CP}^2}
\newcommand{\C}{\mathbb{C}}
\newcommand{\R}{\mathbb{R}}
\newcommand{\Q}{\mathbb{Q}}
\newcommand{\Z}{\mathbb{Z}}
\newcommand{\N}{\mathbb{N}}
\newcommand{\F}{\mathscr{F}}

\begin{abstract}
Given a relatively minimal fibration $f: S \to \Bbb P^1$, defined
on a rational surface $S$, with a general fiber $C$ of genus $g$,
we investigate under what conditions the inequality $6(g-1)\le
K_f^2$ occurs, where $K_f$ is the canonical relative sheaf of $f$.
We give sufficient conditions for having such inequality,
depending on the genus and gonality of $C$ and the number of
certain exceptional curves on $S$. We illustrate how these results
can be used for constructing fibrations with the desired property.
For fibrations of genus $11\le g\le 49$ we prove the inequality:
$$ 6(g-1) +4 -4\sqrt g \le K_f^2.$$

\end{abstract}
\maketitle

\section{Introduction}

Let $S$ be a projective complex nonsingular surface. A fibration
$f:S \to X$ on $S$ is a morphism onto a projective curve $X$ with
connected fibers. Throughout this paper $X$ will be equal to $\Bbb
P^1$. We use the usual identification of divisors and their
associated sheaves and write, for example, $H^i(L)$ instead of
$H^i(S,\mathcal{O}_S(L))$.
\\

Given a fibration $f$, the relative canonical sheaf of $f$ is
defined as $K_f= K_S\otimes f^* K_X^{-1}$, if $X$ is rational then
$K_f=K_S(2C)$, with $C$ a general fiber (which is assumed to have
genus $g \ge 2$). This sheaf is known to be big and nef for
non-isotrivial and relatively minimal fibrations, a result proved
in the foundational papers by Arakelov and Parshin (see
\cite{arakelov} and \cite{parshin}). In the case when $S$ is
rational and $g>0$, $K_f$ is nef even if we drop the hypothesis on
isotriviality. This is due to the fact that in this case $K_f$ is
an effective divisor, thus $K_f\cdot E<0$ would imply that $E$ is
a vertical $(-1)$-curve which gives a contradiction with the
relatively minimal hypothesis (see proof of \ref{primerlema}).
\\

The basic numerical invariant associated to a non-isotrivial
fibration is the so called {\it slope} of $f$ defined as:

$$ \lambda_f= \frac{K_f^2}{\deg(f_*K_f)}.$$

In the case of our interest, when $S$ is a rational surface,
$deg(f_*K_f)=g$ (see Lemma  \ref{lema2.2}) for any fibration $f$
and $\lambda_f=K_f^2/g$.
\\

The study of  the restrictions that the slope of a fibration must
satisfy in relation to the genus $g$ is a central issue in the
theory. As noted before, in the case of rational surfaces the
study of $\lambda_f$ is equivalent to that of $K_f^2$.
\\

This paper is devoted to the relation between $K_f^2$ and $g$ in
the case of rational surfaces. Beside its importance for the study
of the slope, this relation is relevant for another problem, the
bounding of the minimal number $\sigma$ of singular fibers that a
semi-stable non isotrivial fibration must have.
\\

The strict canonical inequality (see \cite{tan} and \cite{vojta}) states
that:

$$ K_f^2 < (\sigma -2)(2g-2),$$
for any semi-stable, non isotrivial fibration $f: S \to \Bbb
P^1$ of genus $g\ge 2$.
\\

In this way inequalities of the sort of $n(g-1)\le K_f^2$, for
some integer $n$, lead to lower bounds for the number $\sigma $
(see \cite{tan} and \cite{tan-tu-zamora}). For instance, the
inequality

$$ 6(g-1)\le K_f^2,$$

\noindent implies, for semistable and non isotrivial fibrations,
that $\sigma \ge 6$.
\\

For the case of surfaces of non-negative Kodaira dimension it is
known that in fact the inequality $ 6(g-1)\le K_f^2$ (and in
consequence $\sigma \ge 6$) holds for any semi-stable and
non-isotrivial fibration (see \cite{konno} and
\cite{tan-tu-zamora}).
\\

However, it is a hard problem to determine for which fibrations on
a rational or ruled surface the inequality $6(g-1)\le K_f^2$ is
valid. Simple examples of  rational surface admitting a fibration
for which $K_f^2< 6(g-1)$ are shown after the statement of Theorem
3.7. In this paper we obtain several general conditions to
guarantee the validity of this inequality for rational surfaces.
\\

Our method is based on the study of the linear systems $|C+nK_S|$
with $n=2,3$. If any of these linear systems is non-empty, then it
is possible to compute its Zariski-Fujita decomposition $P+N$,
with $P$ a nef divisor. The resulting inequality $P^2\ge 0$ gives
inequalities involving $K_f^2$, $g$ and some auto-intersection
numbers of exceptional divisors on $S$. If moreover, $P$ is big
then $\chi(P+K_S) \ge 0$ also gives useful inequalities. Extra
hypotheses on the genus $g$ and the gonality of $C$ allows us to
guarantee the non-emptiness of these linear systems. These
hypotheses are imposed in order to be able to apply Reider's
method to the study of the linear systems. These results are
summarized in Theorems \ref{teorema1} and \ref{teorema2}.
\\

First of all we can compute the negative part $N_1$ of the
Zariski-Fujita decomposition of $C+2K_S$ and the auto-intersection
$N_1^2=-l$. $N_1$ is given by the expression:

$$N_1 = \sum_{i=1}^{s} \big[(l(\Gamma_i)+1)\Gamma_i +
\sum_{j=1}^{\l(\Gamma_i)}(l(\Gamma_i)-j+1)E_{ij} \big],$$

\noindent where $\{\Gamma_i\}$ is the set of $(-1)-$sections of
$f$ and $\sum_{j=1}^{l(\Gamma_i)} E_{ij}$ is a maximal chain
(possibly empty) of vertical $(-2)$ curves satisfying:

\begin{equation*}
\Gamma_i\cdot E_{i1}=1, \quad E_{ik}\cdot E_{im}=\left\{ 
        \begin{tabular}{cc}
             1 \quad &\textrm{ if  $\vert k-m\vert =1$}, \\
             0   &\textrm{otherwise.} \\
        \end{tabular}
\right.
\end{equation*}
\\

Thus, in this notation $l(\Gamma_i)$ is the length of the chain of
$(-2)-$curves attached to $\Gamma_i$, and $s$ the number of $-1$
sections of $f$. The resulting auto-intersection number $l=-N_1^2$
is $l=\sum_i (l(\Gamma_i)+1)$. We call  such  a chain of
$(-2)$-curves a $(-2)$-divisor.
\\

\noindent \textbf{Theorem 3.3} \emph{ Let $f: S \to \Bbb P^1$ be a
non-isotrivial relatively minimal fibration on a rational
surface $S$, with general fiber $C$ of genus $g$. Then the
following statements hold:}
\\

\noindent  \emph{i) If $C+2K_S$ is effective, then $C+2K_S-N_1$ is
nef, and}

\begin{equation*}
0 \leq 4K_f^2-24(g-1)+l.
\end{equation*}

\noindent  \emph{ ii) If $g \geq 7$ and the gonality of $C$ is at
least $4$, then  $C+2K_S$ is effective.}
\\

\noindent  \emph{  iii) If $g \geq 11$ and the gonality of $C$ is
at least 5, then $C+2K_S-N_1$ is also big and}

\begin{equation*}
0 \leq 3K_f^2 -19 (g-1) +l +1.
\end{equation*}

\noindent In particular if $l+1\leq g-1$ then $6(g-1)\leq K_f^2$.
\vspace{.2cm}
\\

We remark that if we contract the $(-1)-$sections we obtain a new
rational surface and an associated pencil having its base locus
just on the image of the $\Gamma_i$. After this contraction the
images of $E_{i1}$ are $(-1)$ curves and can be contracted again
to a non-singular rational surface if we continue this procedure
we can finally contract the divisor $N_1$ and obtain a new
rational nonsingular surface and a pencil on it with its base
locus in the image of the connected components of $N_1$.
Conversely, resolving the base locus of the resulting pencil we
obtain the original fibration on $S$.
\\

In order to explain the content of Theorem \ref{teorema1}, we
introduce a basic example. Start with a pencil generated by two
irreducible, nonsingular plane curves of degree $d$ intersecting
each other transversally and consider the fibration $f: S \to \Bbb
P^1$ obtained by blowing up the $d^2$ base points. The invariants
associated are:

$$g-1=\frac{d(d-3)}{2}, \quad K_f^2=K_S(2C)^2=3d^2-12d+9,$$
and finally, $\Bbb P^2$ being a minimal surface, $l$ coincides in
this case with the number of $(-1)$ sections of $f$, therefore
$l=d^2$.
\\

In this way,

$$ K_f^2-6(g-1)= 9-3d,$$ which is negative for $d>3$. On the other
hand, $C+2K_S$ is effective for $d\ge 6$ (see the computations in Section 4, Example (1)).
\\

Now, if we add the prescribed term in Theorem \ref{teorema1} i),
namely $l/4$ (we are dividing the inequality by $4$), then we
obtain:

$$\frac{d^2}{4}-3d +9,$$ that is in fact positive for all values of $d$.
Note also that in these examples the gonality of the general fiber
is $d-1$.
\\



Returning to the general situation, a similar analysis can be made
for $C+3K_S$. The negative part of $C+3K_S-N_1$ in the
Zariski-Fujita decomposition is of the form $N_1+N_1'+N_2$.
Explicitly the divisors $N_1'$ and $N_2$, are given by:

$$ N_1'=\sum_{j=1}^{\l'(\Gamma_i)}(l'(\Gamma_i)-j+1)E'_{ij},$$

\noindent where $E_i'=\sum_i E'_{ij}$ are maximal $(-2)$ divisors
such that $E'_{i}\cdot C=1$ and $E'_{i}\cdot \Gamma_i=1$, and

$$N_2=\sum_{i=1}^{t} [(m(\Delta_i)+1)\Delta_i +
\sum_{j=1}^{m(\Delta_i)}(m(\Delta_i)-j+1)F_{ij}],$$

\noindent with $\{\Delta_1,..., \Delta_t \}$ the set of $(-1)$
curves on $S$ satisfying $\Delta_i\cdot C=2$, and $F_i=\sum
F_{ij}$ are vertical maximal $(-2)-$divisors such that $F_{i}\cdot
\Delta_i=1$.
\\

Denote $l'=\sum_{i=1}^{s'} (l'(\Gamma_i)+1)$ and $m=\sum_{i=1}^t
(m(\Delta_i)+1)$. Then ${N_1'}^2=-l'$ and $N_2^2=-m$.
\\

With this notation in mind we have:
\\

\noindent \textbf{Theorem 3.7} \emph{Let $f: S \to \Bbb P^1$ a
non-isotrivial relatively minimal fibration on a rational surface
$S$, with general fiber $C$ of genus $g$. Then the following
statements hold:}
\\

\noindent  \emph{ i) If $C+3K_S-N_1$ is effective, then
$C+3K_S-2N_1-N_1'-N_2$ is nef, and}

\begin{equation*}
0 \leq 9K_f^2-60(g-1)+4l+l'+m.
\end{equation*}

\noindent  \emph{ii) If the gonality of $C$ is at least $6$ and $g
\geq 23$, then  $C+3K_S-N_1$ is effective.}
\\

Once again, the divisors $N_1'$ and $N_2$ can be contracted to
obtain a new non-singular surface with a pencil associated to $f$.
If $N_2$ is non-empty the generic element of this pencil will have
in general singularities. These singularities are of nodal type if
the chains $F_{ij}$ are empty.
\\

Thus, in order to produce examples of fibrations satisfying the
hypothesis of Theorems \ref{teorema1} or  \ref{teorema2} we must
start with pencils of curves with a certain bounded kind of
singularities. This is the task in section 4, where we explain how
the previous theorems can be used in order to obtain fibrations
satisfying $6(g-1)\le K_f^2$. The examples are given by blowing up
the base locus of pencils of nodal curves on minimal rational
surfaces.
\\



Before section 4, as a part of a preparatory discussion for Theorem
\ref{teorema2}, we obtain Theorem \ref{teorema3.2}, which gives a general
and uniform bound for the slope of a relatively minimal fibration.
\\

\noindent \textbf{Theorem 3.5}  \emph{ Let $f: S \to \Bbb P^1$ be
a relatively minimal fibration on a rational surface $S$. If the
genus $g$ of the fiber is greater than $11$ and the gonality is at
least $5$, then}

\begin{equation*}
(5+\frac{1}{2})(g-1)-5 \leq K_f^2
\end{equation*}

\noindent \emph{holds.}
\\

Finally, returning to the example on pencils of non-singular plane
curves, note that in order to obtain a positive quantity it is
sufficient to add $3d-9$ to $K_f^2-6(g-1)$. The number $3d-9$ is
approximately equal to $\sqrt g$. In section 5 we find another
class of fibrations satisfying an inequality of the type:

$$ 6(g-1)\le K_f^2 + O(\sqrt g),$$ (denoting by $O(\sqrt g)$ a quantity comparable to $\sqrt
g$). More precisely we prove:
\\

\noindent \textbf{Theorem 5.2} \emph{If $f:S \to  \Bbb P^1$, is a
relatively minimal fibration of genus $11\le g\le 49$ on a
rational surface $S$ such that the gonality of the general fiber
$C$ is at least $5$ and the surface $T$ obtained by blowing-down
the divisor $N_1$ satisfies that $K_T^2 < 0$, then:}

\begin{equation*}
 6(g-1) +4 -4\sqrt g \le K_f^2.
 \end{equation*}

\vspace{.3cm}

Theorem 5.2, together with the previous example gives some
evidence in the sense that probably any relatively minimal
fibration on a rational surface satisfies an inequality of the
type:

$$6(g-1)\le K_f^2 + O(\sqrt g).$$

It should be noticed that Theorem 3.3 is valid on any algebraic
surface satisfying $h^1(S,\mathcal{O}_S)= h^2(S,\mathcal{O}_S)=0$,
and all the results in our paper are still valid if moreover we
assume that $K_T^2\le 9$, with $K_T$ standing for the surface
obtained after contracting the support of the divisor $N_1$.
However, as remarked before, the inequality \newline $6(g-1)\le
K_f^2$ is true for non-isotrivial fibrations on surfaces of
non-negative Kodaira dimension. Obtaining analogous results in the
case of non-rational surfaces of negative Kodaira dimension seems
to be of natural interest.

The authors express their gratitude to anonymous referees for
their valuable comments and critics.

\section{Preliminaries and notation}

We always denote by $f:S \to \mathbb{P}^1$  a relatively minimal
fibration on a rational surface $S$. We will denote by $K_S$ a
canonical divisor of $S$.
\\

$C$ will denote a general fibre of $f$ which is assumed to
have genus $g$, and $K_f=K_S(2C)$ will be the relative canonical
sheaf of $f$.
\\

In order to simplify the notation we shall write:

\begin{equation*}
a= K_f^2 \quad \textrm{and} \quad b=g-1.
\end{equation*}

\noindent Some standard equalities are used systematically:

\begin{equation*}
C \cdot K_f=2b, \quad K_S^2=a-8b.
\end{equation*}

The following Lemma, a result taken from \cite{reider}, will be invoked
several times ($S$ is here an arbitrary surface):

\begin{lem} \label{primerlema} Let $L$ be a nef divisor on $S$.
\\

\noindent i) If $L^2\ge 5$ and $\vert L+K_S\vert=\emptyset $ then there
exists a base point free pencil $|E|$ on $S$ such that either $E \cdot L=
0 \text{ or } 1$.
\\

\noindent ii) If $L^2\ge 10$ and $\vert L+K_S\vert$ does not define a
birational map then there exists a base point free pencil $|E|$ on
$S$ such that either $E \cdot L= 1 \text{ or } 2$. \end{lem}

\begin{proof} ii) is just Corollary 2 in \cite{reider}. Even when i) is not
explicitly stated in \cite{reider} it follows just by the same
argument used in the proof of Corollary 2 in \cite{reider}.
\end{proof}

Now, in part for the sake of completeness, in part for
illustrating the kind of argument that will be used in the sequel,
we state and prove a Lemma that is by now well known (see \cite{konno} and \cite{tan-tu-zamora}).

\begin{lem} \label{lema2.2} Let $f$ be, as before, a relatively minimal
fibration on a rational surface $S$. Suppose that $g>0$, then $\vert C+K_S\vert$ is
effective and nef.
\end{lem}

\begin{proof} We can be more specific about the dimension of the
space of sections of $C+K_S$.
\\

Indeed, the surface $S$ is rational, thus $h^0(K_S)=h^1(K_S)=0$.
By Leray's spectral sequence we have also
$h^0(f_*K_S)=h^1(f_*K_S)=0$. Thus in the decomposition of $f_*K_S$
like a sum of invertible sheaves in $\Bbb P^1$ we must have:

\begin{equation*}
f_{*} K_S= \bigoplus^g \mathcal{O}_{\Bbb P^1} (-1),
\end{equation*}

\noindent  therefore,

\begin{equation*}
h^0(C+K_S)=h^0(f_* K_S \otimes \mathcal{O}_{\Bbb P^1} (1))=g.
\end{equation*}

This proves the first assertion. Now, if $(C+K_S) \cdot E<0$ for
some irreducible curve $E$, then, being $C$ nef, $E$ must be a
vertical $(-1)-$curve (see \cite{ciliberto}, proof of Proposition
4.1) and in consequence $C \cdot E=0$. But this is impossible
because $f$ is relatively minimal. This proves the Lemma.
\end{proof}

\section{Adjoint systems and the slope of $f$}

In this section we investigate the properties of the linear
systems $\vert C+mK_S \vert$ for $m=1,2,3$. We apply their
properties to the study of the slope of $f$.

\begin{lem} \label{lema5} If $b\ge 6$ and the gonality of $C$ is at least $4$,
then $h^0(C+2K_S)>0$ and $a\ge 5b$. \end{lem}

\begin{proof}  This follows easily from Corollary 4.4 of \cite{konno3}.








\end{proof}

If $C+2K_S$ is effective, it admits a Zariski-Fujita decomposition as the  sum of a
nef divisor and a negative part, let us compute this negative part. \\

Let

\begin{equation*}
N_1 = \sum_{i=1}^{s} \big[(l(\Gamma_i)+1)\Gamma_i +
\sum_{j=1}^{\l(\Gamma_i)}(l(\Gamma_i)-j+1)E_{ij} \big],
\end{equation*}

\noindent where $\{ \Gamma_1,...,\Gamma_s \}$ is the set of $(-1)-$sections of
$f$ and $\sum E_{ij}$ is a maximal chain (possibly empty) of vertical $(-2)$ curves
satisfying:

\begin{equation*}
\Gamma_i \cdot E_{i1}=1, \quad E_{ik} \cdot E_{im}=\left\{ 
        \begin{tabular}{cc}
             1 \quad &\textrm{ if  $\vert k-m\vert =1$}, \\
             0    &\textrm{otherwise,} \\
        \end{tabular}
\right.
\end{equation*}

\noindent and $l(\Gamma_i)$ will denote the length of the maximal
chain $\sum E_{ij}$. If the chain is empty, then we convey that
$l(\Gamma_i)=0$. In the language of \cite{barth}, this is
expressed by saying that the divisor $E_i=\sum E_{ij}$ is a
$(-2)-$curve and $\Gamma_i \cdot E_i=1$. We prefer to call  such a
chain a $(-2)-$divisor, in order to distinguish it from the
irreducible case. It should be noticed that there is only one
maximal connected $(-2)-$divisor (possibly empty) attached to each
$(-1)-$section.
\\

We will also use  the notation $\Gamma=\sum_{i=1}^s \Gamma_i$.
\\

\begin{lem} \label{lemaN1} The negative part of $C+2K_S$ in the Zariski-Fujita decomposition is $N_1$ and $N_1^2=-l$, where $l=\sum_{i=1}^s (l(\Gamma_i) +1)$.

\end{lem}

\begin{proof}
Although our calculus relies on the argument used in
\cite{ciliberto} Proposition 4.1, it is slightly different and
uses the Zariski-Fujita algorithm in order to obtain a more
explicit expression.
\\

If $(C+2K_S)\cdot D< 0$ for some irreducible curve $D$, then being
$C$ nef and $D^2<0$ we must have that $D$ is a $(-1)$-curve such
that $C\cdot D\le 1$. Since $f$ is relatively minimal we conclude
that $D$ is a $(-1)$-section.
\\

Next we need to find the
irreducible curves $D$ such that:

\begin{equation} \label{equation1}
(C+2K_S-\Gamma)\cdot D<0.
\end{equation}

Since $D\ne \Gamma_i$ we have that it is not a $(-1)-$ curve,
therefore $D\cdot K_S\ge 0$ and $D\cdot \Gamma >0$.
\\

Let $T$ be the surface obtained by contracting $\Gamma$. Then in
$T$, (\ref{equation1}) becomes:

\begin{equation*}
(C_0 + K_T)\cdot \pi_* D< 0,
\end{equation*}

\noindent with $\pi$ standing for the contraction and $C_0\in
\vert \pi (C)\vert$.
\\

Just by the same argument as before we get that $\pi_* D$ is a
$(-1)-$curve with $0\le C_0 \cdot \pi_* D\le 1$. In particular
$\Gamma \cdot D\le 1$. Thus we obtain that $D\cdot \Gamma=1$,
$K_S\cdot \Gamma=0$ and $C\cdot D=0$. In this way $D=E_{i1}$ for
some $i$.
\\

Now, the Zariski-Fujita algorithm commands solving the system of
equation in $\alpha_j$, $\beta_k$:

\begin{align*}
 (C+2K_S -\Gamma)\cdot \Gamma_i & =(\sum \alpha_j \Gamma_j + \sum
\beta_k E_{k1})\cdot \Gamma_i, \\
(C+2K_S -\Gamma)\cdot E_{i1} & = (\sum \alpha_j \Gamma_j + \sum
\beta_k E_{k1})\cdot E_{i1}, \end{align*}

\noindent for $i= 1,...s$. It is easy to see that the solution is $\alpha_j=\beta_k=1$ for
all $j$ and $k$. So in this step we need to subtract $\Gamma +
\sum E_{i1}$ to $C+2K_S-\Gamma$.
\\

The rest of the computation is iterative. Assume that on some step
of the algorithm we have obtained the divisor:

\begin{equation*}
 C+ 2K_S -N,
\end{equation*}

\noindent with $N= \sum_i [(l_i+1) \Gamma_i +\sum_{j=1}^{l_i} (l_i
-j+1) E_{ij}].$
\\

Let $T$ now be the surface obtained by contracting the support of
$N$. We obtain, with the analogous notation:

\begin{equation*}
 (C_0+K_T)\cdot \pi_* D< 0,
\end{equation*}

\noindent therefore, as before, $C\cdot D=0$ and $N\cdot  D=1$. So, $D$ must
be equal to $E_{i,l_i+1}$ for some $i$. The Zariski-Fujita
algorithm leads again to subtract $\sum_i [ \Gamma_i
+\sum_{j=1}^{l_i+1} E_{ij}]$.  Therefore, the negative part is:

\begin{equation*}
N_1 = \sum_{i=1}^{s} \big[(l(\Gamma_i)+1)\Gamma_i +
\sum_{j=1}^{\l(\Gamma_i)}(l(\Gamma_i)-j+1)E_{ij} \big].
\end{equation*}

Now we will calculate $N_1^2$:

\begin{align*}
N_1^2&= \Big(\sum_{i=1}^{s} \big[(l(\Gamma_i)+1)\Gamma_i +
\sum_{j=1}^{\l(\Gamma_i)}(l(\Gamma_i)-j+1)E_{ij} \big] \Big)^2\\
&=  \sum_{i=1}^{s} \big[(l(\Gamma_i)+1)\Gamma_i +
\sum_{j=1}^{\l(\Gamma_i)}(l(\Gamma_i)-j+1)E_{ij} \big]^2 \\
&(\textrm{because $\Gamma_i \cdot \Gamma_k= \Gamma_k \cdot E_{ij}=E_{kj} \cdot E_{ij}=0$ if $k \neq i$})\\
&=\sum_{j=1}^{\l(\Gamma_i)} \big[ -(l(\Gamma_i)+1)^2+2(l(\Gamma_i)+1)l(\Gamma_i)+(\sum_{j=1}^{l(\Gamma_i)}(l(\Gamma_i)-j+1)E_{ij})^2 \big] \\
&(\textrm{because $\Gamma_i \cdot E_{ij}=0$ if $j > 1$})\\
&= \sum_{j=1}^{\l(\Gamma_i)} \big[ -(l(\Gamma_i)+1)^2+2(l(\Gamma_i)+1)l(\Gamma_i)-2\sum_{j=1}^{l(\Gamma_i)-1}(l(\Gamma_i)-j+1)-2\big]\\
&(\textrm{because $(\sum_{j=1}^{l(\Gamma_i)}(l(\Gamma_i)-j+1)E_{ij})^2=\sum_{j=1}^{l(\Gamma_i)}(-2(l(\Gamma_i)-j+1)^2)+2\sum_{j=1}^{l(\Gamma_i)-1}(l(\Gamma_i)-j+1)(l(\Gamma_i)-j)$})\\
&=\sum_{j=1}^{\l(\Gamma_i)} \big[-(l(\Gamma_i)+1)(l(\Gamma_i)+1-2l(\Gamma_i)+l(\Gamma_i))\big]=-l.
\end{align*}

\end{proof}



After these computations we have:

\begin{teo} \label{teorema1} Let $f: S \to \Bbb P^1$ be a relatively
minimal fibration on a rational surface $S$, with general fiber
$C$ of genus $g$. Then the following statements hold:

i) If $C+2K_S$ is effective, then $C+2K_S-N_1$ is nef, and

\begin{equation*}
0 \leq 4K_f^2-24(g-1)+l.
\end{equation*}

ii) If $g \geq 7$ and the gonality of $C$ is at least $4$, then  $C+2K_S$ is effective.

iii) If $g \geq 11$ and the gonality of $C$ is at least
5, then $C+2K_S-N_1$ is also big and

\begin{equation*}
0 \leq 3K_f^2 -19 (g-1) +l +1.
\end{equation*}

\noindent In particular if $l+1\leq g-1$ then $a\geq 6b$.
\end{teo}

\begin{proof} Parts i) and ii) follow from
Lemma 3.2, the inequality in i) merely expresses the fact that
$(C+2K_S-N_1)^2\ge 0$.
\\

We proceed with the proof of iii). We will prove that $\vert C +
2K_S\vert$ defines a birational map. This would imply that the nef
part of the divisor is big (see \cite{badescu}, 14.18).
\\

Assume, for contradiction, that $\vert C + 2K_S\vert$ does not
define a birational map.We know, by Lemma \ref{lema5}, that $a\ge 5b$,
thus, $(C+K_S)^2=a-4b\ge b \ge 10$, since $g\ge 11$. By Lemma \ref{primerlema}
ii), $S$ admits a base point free pencil $\vert E\vert$ with
$E \cdot (C+K_S)=1 \text{ or } 2$. If, for instance, $E \cdot (C+K_S)=1$, then

$$ 1=K_S \cdot E+C \cdot E=2g_E-2+C \cdot E.$$

The only possibility for this equality holding is $g_E=0$ and
$C \cdot E=3$, but then $C$ must be trigonal. Similarly, $E \cdot (K_S+C)=2$
implies that $C$ is tetragonal or hyperelliptic. The last
assertion follows from Mumford's vanishing theorem:

\begin{align*}
 0&\le \chi(C+3K_S-N_1)=h^0(C+3K_S-N_1)\\
&=\frac{(C+2K_S-N_1) \cdot (C+3K_S-N_1)}{2} + 1=3K_f^2-19(g-1)+l+1.
\end{align*}

\end{proof}

Now, we can make a similar analysis for the linear system $\vert
C+3K_S\vert$. We start by proving the following bound for $l$:

\begin{lem}   \label{lemal}  Assume $g\ge 11$ and the gonality of $C$ is
at least $5$, then

$$l\le \frac{5}{2}b + 14.$$
\end{lem}

\begin{proof} Let $\pi:S \to T$ be the contraction of the divisor
$(N_1)_{\text{red}}$. After contracting $\Gamma_i$, $E_{i1}$
is contracted to a $(-1)$ curve, and the same is valid for
$E_{i,j+1}$ after contracting $E_{i,j}$, so we see that $T$ is a
nonsingular surface.
\\

Moreover,

$$K_S= \pi^* K_T + N_1,$$
and,

$$ K_S^2= (\pi^* K_T + N_1)^2= a-8b.$$

Thus, $\pi^* K_T^2=a-8b +l$. Combining this with part iii) of
Theorem \ref{teorema1}, we obtain the equality:

$$ h^0(C+3K_S-N_1)-1 + 2l -3\pi^* K_T^2=5b.$$

Using $K_T^2\le 9$ we obtain the desired inequality.
\end{proof}

So far we have obtained, under the hypothesis of Theorem  \ref{teorema1}, a bound for the slope of $f$:

\begin{teo} \label{teorema3.2} Let $f: S \to \Bbb P^1$ be a relatively minimal fibration on a rational surface $S$. If the genus $g$ of the fiber is greater than or equal
to $11$ and the
gonality of $C$ is at least $5$ then
\begin{equation*}
(5+\frac{1}{2})(g-1)-5 \leq K_f^2.
\end{equation*}
\end{teo}

\begin{proof} The statement is just a combination of Theorem  \ref{teorema1} and Lemma  \ref{lemal}.
\end{proof}

We return to the analysis of the linear system $|C+3K_S|$. The following is analogous to Lemma  \ref{lema5}.

\begin{prop} \label{prop3.5}  If $b\ge 22$ and the gonality of $C$ is at
least 6, then $\vert C+3K_S -N_1 \vert \ne \emptyset$.
\end{prop}

\begin{proof} The hypotheses of Theorem \ref{teorema1}, iii) are satisfied, therefore $C+2K_S-N_1$ is nef.
 Note that, by Lemma \ref{lemal} and part iii) of Theorem  \ref{teorema1}

\begin{align*}
(C+2K_S -N_1)^2 &=4a-24b+l \\ &\ge \frac{4}{3}(19b-l-1)-24b+l \\
&= \frac{1}{3}(4b-l-4)\ge \frac{4}{3}b-\frac{5}{6}b-6=\frac{1}{2}b
-6.
\end{align*}

Thus, if $b\ge 22$ then $(C+2K_S -N_1)^2\ge 5$. By Lemma \ref{primerlema} i) if
$\vert C+3K_S -N_1 \vert= \emptyset$, then there exists a base
point free pencil $\vert E \vert$ such that:

$$E \cdot (C+2K_S-N_1)= 1 \text{ or } 0.$$

We have either that $E$ is contracted by $\vert C+2K_S -N_1
\vert$ or $E$ is a rational curve with $E^2=0$ and
$E \cdot (C+2K_S-N_1)= 1 $.
\\

In the first case $E$ is contracted as well by $C+2K_S$, but
$C+K_S$ is nef and $(C+K_S)^2=a -4b\ge 10$. Thus, by part ii) of
Lemma \ref{primerlema}, there exists a pencil $\vert E'\vert$, such that:

\begin{equation*}
(C+K_S) \cdot E'= 2 \text{ or } 3.
\end{equation*}

\noindent But then:

\begin{equation*}
C \cdot E'-2\le 2 \text{ or } 3,
\end{equation*}

\noindent which gives a contradiction with
the assumption on the gonality of $C$.
\\

On the other hand, if $E$ is rational then $E \cdot K_S=-2$ and:

\begin{equation*}
E \cdot C-4-N_1 \cdot E= 1.
\end{equation*}

Consider the $\Bbb P^1-$ fibre bundle associated with $E$ (see
\cite{barth}, V 4.3):

\begin{align*}
\xymatrix { S \ar[r]^{\phi} & R  \ar[r] & \Bbb P^1. \\}
\end{align*}

$R$ is a Hirzebruch surface $\Bbb F_n$ and, if $n\ne 1$ then $\Bbb
F_n$ is minimal and $N_1$ must be contracted by $\phi$. We
conclude that $N_1$ is a sum of vertical curves with respect to
$\vert E \vert$, and $N_1 \cdot E=0$.
\\

If $R=\mathbb{F}_1$ then we could have $N_1 \cdot E=1$, but in this case $C \cdot E=6$ and the image of $C$
in $\mathbb{F}_1$ becomes equivalent to $\Gamma_0 + 6E$, with $\Gamma_0$ denoting the (-1)-section,
that is, the image of $N_1$ under $\phi$. A simple calculus using the adjunction formula on $\mathbb{F}_1$
shows that $b=17$. Therefore, the only possibility is $N_1 \cdot E=0$ and $E \cdot C \leq 5$.
This final contradiction proves the Proposition. \end{proof}

Next we need to obtain the negative part of $C+3K_S-N_1$, the
computation is quite analogous to that of the negative part of
$C+2K_S$, using in this case that $C+2K_S-N_1$ is nef. This
negative part is $N_1+N_1'+N_2$ with

\begin{equation*}
N_1'=\sum_{j=1}^{\l'(\Gamma_i)}(l'(\Gamma_i)-j+1)E'_{ij},
\end{equation*}

\noindent where $E_i'=\sum_i E'_{ij}$ are maximal $(-2)$ divisors
such that $E'_{i} \cdot C=1$ and $E'_{i} \cdot \Gamma_i=1$, of
length $\l'(\Gamma_i)$; and

$$N_2=\sum_{i=1}^{t} [(m(\Delta_i)+1)\Delta_i +
\sum_{j=1}^{m(\Delta_i)}(m(\Delta_i)-j+1)F_{ij}],$$

\noindent with $\{\Delta_1,..., \Delta_t \}$ the set of $(-1)$
curves on $S$ satisfying $\Delta_i \cdot C=2$, and $F_i=\sum
F_{ij}$ are vertical maximal $(-2)-$divisors such that $F_{i}
\cdot \Delta_i=1$ and of length $m(\Delta_i)$. To each
$(-1)-$section $\Gamma_i$ there is associated a unique (possible
empty) divisor $E_i'$ and the same is true for divisors $\Delta_i$
with respect to the chains $F_i$.
\\

Denote $l'=\sum_{i=1}^{s'} (l'(\Gamma_i)+1)$ and $m=\sum_{i=1}^t
(m(\Delta_i)+1)$.
\\

As in the case of $N_1^2$, a similar calculus shows that
${N_1'}^2=-l'$ and $N_2^2=-m$. We have obtained, in analogy with
Theorem \ref{teorema1}:

\begin{teo} \label{teorema2} Let $f: S \to \Bbb P^1$ be a relatively
minimal fibration on a rational surface $S$, with general fiber
$C$ of genus $g$. Then the following statements hold:

i) If $C+3K_S-N_1$ is effective, then $C+3K_S-2N_1-N_1'-N_2$ is nef, and

\begin{equation*}
0 \leq 9K_f^2-60(g-1)+4l+l'+m.
\end{equation*}

ii) If the gonality of $C$ is at least $6$ and $g \geq 23$, then  $C+3K_S-N_1$ is effective.

\end{teo}

The proof of i) follows, as in the analogous case in Theorem
3.2, from the fact that $(C+3K_S-2N_1-N_1'-N_2)^2\ge 0$, since the
divisor is nef. Part ii) follows from Proposition \ref{prop3.5}.
\\

In particular, under the hypothesis of Theorem \ref{teorema2} i), if
$4l + l'+m\le 6(g-1)$, then $ 6(g-1) \le K_f^2$.

\section{Examples}

In the following examples we start with a surface $T$ and a pencil
$W$ of nodal curves on $T$. The fibration will be obtained in a
surface $S$ by blowing-up the base locus of the pencil. The
general fiber of the fibration will be denoted by $C$ and will
correspond with the proper transform of the general element
$C_0\in W$.
\\

In general it is not true that blowing up the base locus of a
pencil on a minimal surface $T$ gives rise to a relatively minimal
fibration. The simplest example comes from considering the pencil
generated by an  irreducible conic $Q$ in $\Bbb P^2$ and the
product of two lines $L_1$, $L_2$ that intersects in a point not
contained in $Q$. After blowing up the base locus of this pencil
the proper transform of both, $L_1$ and $L_2$ become vertical
$(-1)-$curves. More complicated examples can be constructed.
Fortunately if we limit the singularities of the general element
of the pencil we can gain control of the situation.
\\

Thus, we start by computing the conditions to have in the
following examples relatively minimal fibrations. Consider a curve
$C_0=D_0+D_1$ in the pencil $W$  and letting $p_1,...,p_l$ be the
nonsingular points in the base locus of $W$ and letting
$q_1,...,q_m$ be the nodal points in the base locus of $W$, we
will use the following notation:

\begin{align*}
l_{01}&=\#\{p_k: p_k \in D_0\}\\
l_{02}&=\#\{q_k: \textrm{$q_k$ is simple for $D_0$} \}\\
m_{0}&=\#\{q_k: \textrm{$q_k$ is node for $D_0$}\}\\
l_0&=l_{01}+l_{02}.
\end{align*}

\noindent Suppose that the proper transform  $\widetilde D_0
\subset S$ is a vertical (-1)-curve.  We have the following
equations:
\\

\noindent (i)   $\widetilde D^2_0=-1=D^2_0-\ell_{01}-\ell_{02}-4m_0,$

\noindent (ii) $g_{\widetilde D_0}=0$, which implies that
$\frac{D_0\cdot (D_0+K_T)}{2}-m_0=-1,$

\noindent (iii) $\widetilde D_0\cdot \widetilde C_0=0$, that is,
$D_0\cdot C_0-\ell_{01}-2\ell_{02}-4m_0=0$.

\vskip2mm

\noindent  From (i) and (ii) we have that
$D^2_0-\ell_{01}-\ell_{02}+1=2D^2_0+2D_0\cdot K_T+4$, and
$D^2_0+2D_0\cdot K_T+3=-(\ell_{01}+\ell_{02})$, that is,  if
$D_0^2+2D_0\cdot K_T+3>0$  curve $D_0$ cannot exist. On the other
hand
 we have  from (i) and (iii) that $D_0\cdot
C_0-\ell_{01}-2\ell_{02}=D_0^2-\ell_{01}-\ell_{02}+1$, then
$D_0\cdot C_0-D^2_0=\ell_{02}+1>0$. This implies that if there
exists curve $D_0$ the following two conditions hold
simultaneously:

$$D^2_0+2D_0\cdot K_T<0,\hskip4mm D_0\cdot C_0-D^2_0>0.$$

For the case $T=\Bbb P^2$ we have that:

\begin{align*}
\widetilde{D_0}^2&=d_0^2-l_0-4m_0=-1\\
g_{\widetilde{D}_0}-1&=\frac{d_0(d_0-3)}{2}-m_0=-1\\
C_0 \cdot \widetilde{D}_0&=d_0d-l_{01}-2l_{02}-4m_0=0.
\end{align*}

\noindent where $d_0$ is the degree of $D_0$. From the above
equations we can conclude that $d_0(d_0-6)+3=-l_0$ and
$d_0(2d_0-d) > 0$. If $d_0 \geq 6$ then we have a contradiction
with the first equation and if $d_0 < 6$ then  $d < 12$. In
conclusion a nodal pencil of degree $d \geq 12$ in $\Bbb P^2$
gives place to a relatively minimal fibration.
\\

Similar computations show that if  $T=\mathbb{F}_0$ and the class of $C_0$ is $(\alpha, \beta)$ with $\alpha, \beta \geq 8$ then the associated fibration is relatively minimal.

\vspace{.3cm}

\begin{enumerate}

\item We want to use Theorem \ref{teorema2} i) in order to construct a pencil of
plane curves  such that the associated fibration satisfies
$6(g-1)\le K_f^2$.
\\

Consider  positive integers $l,m,d$ satisfying the conditions $l+4m=d^2$ and $l=2m$, i.e., $6m=d^2$. Fix $m$ points $q_1,...,q_m \in \Bbb P^2$, if $V(q_1,...,q_m)$ denotes the
linear system of plane curves of degree $d$ having nodes at $q_i$, then:

\begin{equation*}
\dim V(q_1,...,q_m) \geq  \frac{d(d+3)}{2}-3m= \frac{l-2m+3d}{2}=\frac{3d}{2},
\end{equation*}

\noindent which is positive. Taking a general pencil contained in $V(q_1,...,q_k)$, and blowing
up its base locus, we will obtain a rational fibered surface $S$.
\\

For $d > 9$ the condition $|C+3K_S| \neq \emptyset$ is satisfied
because:

\begin{align*}
K_S&=[-3,1,1,1,...,1] \\
C&=[d, -1,...,-1,-2,...,-2].
\end{align*}

\noindent The above equalities mean equality of classes in the
Picard group of $S$, where we use the standard ordered basis for
$Pic (S)\simeq Pic (\Bbb P^2) \oplus_i \Gamma_i\Bbb Z \oplus_j
\Delta_j \Bbb Z$ with $\Gamma_i$ the exceptional divisors
associated with $p_i$ and $\Delta_j$ the exceptional divisors
associated with $q_j$, and denote by $H$ the hyperplane class in
$\Bbb P^2$. Therefore

\begin{equation*}
C+3K_S=[d-9,2,...,2,1,...,1].
\end{equation*}

\noindent Note that the decomposition of the divisor
$C+3K_S=(d-9)H+(\sum 2\Gamma_i + \sum \Delta_i)$ is the
Zariski-Fujita decomposition of $C+3K_S$. In this case, being
$\Bbb P^2$ a minimal surface the divisor $N_1$ defined in the
previous section is just $\sum \Gamma_i$ and $N_2= \sum \Delta_i$.
The numbers $l$ and $m$ coincide with the previously defined in
section 3. By the genus formula $6b=3d^2-9d-6m$, in order to be in
the hypothesis of the remark after Theorem \ref{teorema2} we need
the inequality $4l+m \leq 6b$ which is equivalent to  $3d \leq m$.
If we take $d = 18, m=54$ we obtain a curve of genus $82$ , so
with this numerical conditions we have a fibration satisfying $6b
\leq a$. Moreover the gonality of $C$ is $d-2= 16$ (see
\cite{kato}).

\vspace{.5cm}

\item Consider the Hirzebruch surface $T=\mathbb{F}_0$ and  let $C_0$ be an effective divisor on $T$ of class $(\alpha,\beta)$.
The arithmetic genus of $C_0$ is $g_a=1 + \frac{C^2_0+ C_0\cdot
K_T}{2}$, i.e., $g_a-1=\alpha \beta -\alpha - \beta$. Suppose also
that $C_0$ has $m$ nodes. Then its geometric genus $g$ is $\alpha
\beta -\alpha - \beta + 1 -m$. The classes of $K_S$ and $C$ in
$Pic(S)$ are given by:

\begin{align*}
K_S&=[(-2,-2),1,1,...,1] \\
C&=[(\alpha, \beta),-1,...,-1,-2,...,-2].
\end{align*}

Here we use again the standard ordered basis for $Pic (S)\simeq
Pic (\Bbb F_0) \oplus_i \Gamma_i\Bbb Z \oplus_j \Delta_j \Bbb Z$
with $\Gamma_i$ the exceptional divisors associated with $p_i$ and
$\Delta_j$ the exceptional divisors associated with $q_j$. The
divisors $N_1$ and $N_2$, as defined in the previous section are
given again in this example by $N_1=\sum \Gamma_i$ and $N_2= \sum
\Delta_i$.

\noindent Since $h^0(T, \mathcal{O}_T(C_0))-1 \geq \frac{C_0^2-C_0
\cdot K_T}{2}$, we have that $\dim |C_0| > 0$ if and only if
$\alpha \beta+ \alpha + \beta > 3m$. In order to apply Theorem
\ref{teorema2} we need  $|C+3K_S-N_1| \neq \emptyset$, for this,
it  is enough to have $\alpha \geq 7$ and $\beta \geq 7$, because
$3K_S+C=[(\alpha- 6, \beta - 6), -2,-2,-1,...,-1]$. In order to
guarantee that the resulting fibration is relatively minimal we
must take both $\alpha, \beta \geq 8$.
\\

We also need to have $4l+m \leq 6b$. Since $2 \alpha \beta = C_0^2=l +4m$,  then $8 \alpha \beta -15 m \leq 6 \alpha \beta - 6 \alpha - 6 \beta - 6m $,
therefore  $3m \geq \frac{2}{3} \alpha \beta + 2 \alpha + 2 \beta$.  Then the condition $3m \in [\frac{2}{3} \alpha \beta + 2 \alpha + 2 \beta, \alpha \beta+ \alpha + \beta]$
is satisfied taking for example $\alpha = \beta=8$ and $m=26$. Then if we blow-up the nodes in the pencil given by $C_0$ we obtain the following diagram:

\begin{align*}
\xymatrix {
           S \ar[dr]_f \ar[r] & T \ar[d] \\
                                      &  \mathbb{P}^1,}
                                      \end{align*}

\noindent where $f$ is a fibration of genus $23$ that satisfies $6b \leq a$.

\vspace{.5cm}

\item In this example we exhibit a fibration for which the equality  $K_f^2=6(g-1)$ holds.
Following the notation of the above example consider the numbers
$\alpha, \beta, l$ and $m$ satisfying  $3m < \alpha \beta+ \alpha
+ \beta$, $l+4m=2 \alpha \beta$, since
$a=8(\alpha-1)(\beta-1)-l-9m$ and $b=\alpha \beta - \alpha
-\beta-m$, the condition $a=6b$ is equivalent to having $2\alpha +
2\beta = m+8$. For example the numbers $\alpha=\beta=8$, $m=24$
satisfy the above conditions. In this case we obtain a fibration
of genus $25$. We must observe that the minimal degree of a map
$C_0 \to \Bbb P^1$ is $\alpha$ but we can not guarantee that
$\alpha$ is the gonality of its normalization $C$, it could be in
fact lower (see \cite{martens}).

\end{enumerate}

\section{Slope of fibration with $11 \leq g \leq 49$}

Let as before $f:S \to \Bbb P^1$ be a relatively minimal fibration on a rational surface. Through this section assume morever that $g\ge 11$ and the
gonality of $C$ is at least $5$.
\\

Using Lemma \ref{lema2.2} and Lemma  \ref{lema5}
we have that $K_f$ is big and nef.
Thus:

\begin{align*}
h^0(nK_f) & = \frac{nK_f((n-1)K_f+2C)}{2}+1 \\
& =\frac{n(n-1)a +2bn}{2}+1.
\end{align*}

\noindent On the other hand, consider the direct image sheaves:

\begin{equation*}
 f_* nK_f= \bigoplus_{i=1}^{b(2n-1)} \mathcal{O}(a_i).
 \end{equation*}

\noindent Note that by Mumford's vanishing
$h^1(nK_f-(n+1)C)=h^1((n-1)(C+K_S)+K_S)=0$ for $n\ge 2$, as
$C+K_S$ is nef. Thus, by the projection formula

\begin{align*}
a_i-(n+1)&\ge -1,\\
a_i&\ge n.
\end{align*}

In this way $ 2(n+1)(2n-1)b\le 2h^0(f_* nK_f)=2h^0(nK_f)= n(n-1)a
+ 4bn+2$. Substraction of both terms leads to the conclusion that
$q(x)= (a-4b) x^2 -(a-2b)x +2b+2$ is positive when evaluated in
any integer $n\ge 2$.
\\

Some properties of this polynomial are summarized in:

\begin{prop}  \label{prop5.2} Let $f:S \to \Bbb P^1$ be a semistable,
non isotrivial fibration on a rational surface $S$. Then:
\\

i) $q$ is positive when evaluated in any integer $n$.
\\

ii) If $q$ has real roots then they are located in $(0,1)\cup
(1,2)$.
\\

iii) The discriminant of $q$ is $\Delta_q= (a-6b)^2-8(a-4b)$.
\\

iv) If $\Delta_q \leq 0$ then:

$$ 6(g-1)+4-4\sqrt {g} \le K_f^2.$$
\end{prop}

\begin{proof} i) By construction $q(n)>0$ for $n\ge 2$. It is easy
to verify that $q(1)=2$ and $q(0)=2b+2$. Moreover, the critical
value of $q$ is $\frac{2(a-4b)}{a-2b}$ which turns out to satisfy

$$ 0< \frac{2(a-4b)}{a-2b}< 2.$$

\noindent This proves i) and ii). Part iii) is just a direct computation.
\\

For iv), consider the discriminant $\Delta_q$ as a polynomial in
$a$:

$$ \Delta_q(a)= a^2 -(12b+8)a + 36b^2 +32 b.$$

\noindent If $\Delta_q(a)\le 0$ then $a$ is in the ``negative region" of this
quadratic function. Thus,

$$\frac{12b+8 -\sqrt{(12b+8)^2-4(36b^2+32b)}}{2}\le a.$$

This proves the Proposition.
\end{proof}

The next Theorem gives an inequality for the slope of fibrations of
genus $11\le g \le 49$, with an extra assumption on the surface
$T$ obtained after blowing down the negative part of $C+2K_S$.

\begin{teo} If $f:S \to  \Bbb P^1$ is a relatively
minimal fibration of genus $11\le g\le 49$ on a rational surface $S$ such that the gonality
of the general fiber $C$ is at least $5$ and the surface $T$
obtained by blowing-down the divisor $N_1$ satisfies that
$K_T^2 < 0$, then:

$$ 6(g-1) +4 -4\sqrt g \le K_f^2.$$
\end{teo}

\begin{proof} We can assume $a\leq 6b$. Following the proof of Lemma \ref{lemal}  and by Theorem  \ref{teorema1}  (iii) we have $2\ell\leq 5b+1+3\pi^*(K^2_T)$, the condition $K^2_T<0$ implies
$\ell+1\leq \frac{5}{2}b$, since $3a\geq 19b-\ell-1$ we obtain $\frac{11}{2}b\leq a$. Hence we have:

$$\Delta_q= (a-6b)^2 -8(a-4b) \le
(1/2)^2b^2-12b=\frac{b(b-48)}{4}.$$

Combining with part iv) of Proposition \ref{prop5.2} we get the theorem.
\end{proof}

\noindent \textbf{Claudia R. Alc\'antara:}
\noindent Departamento de Matem\'aticas, Universidad de Guanajuato;
\noindent Jalisco s/n,  Valenciana,  C.P. 36240, Guanajuato, Guanajuato, M\'exico.
\noindent claudia@cimat.mx

\noindent \textbf{Abel Castorena:}
\noindent Instituto de Matem\'aticas, Unidad Morelia, Universidad Nacional Aut\'onoma de M\'exico;
\noindent Apartado Postal 61-3 (Xangari), C.P. 58089, Morelia, Michoac\'an, M\'exico.
\noindent abel@matmor.unam.mx

\noindent \textbf{Alexis G. Zamora:}
\noindent Unidad Acad\'emica de Matem\'aticas, Universidad Aut\'onoma de Zacatecas;
\noindent Camino a la Bufa y Calzada Solidaridad C.P. 98060, Zacatecas, Zac., M\'exico.
\noindent alexiszamora06@gmail.com


\begin{thebibliography}{99}


\bibitem{arakelov} Arakelov S.J.:
\emph{Families of algebraic curves with fixed degeneracies}. Math. USSR-Izv. 5, 1971.


\bibitem{badescu} Badescu L.:
\emph{Algebraic Surfaces}. Springer-Verlag, 2001.


\bibitem{barth} Barth W., Peters C., Van de Ven A.:
\emph{Compact Complex Surfaces}. Springer Verlag, 1984.






\bibitem{kato} Coppens M., Kato T.:
\emph{The gonality of smooth curves with plane models.}  Manuscripta math. 70, 5-25, 1990.


\bibitem{ciliberto} Calabri A., Ciliberto C.:
\emph{Birational classification of curves on rational surfaces.}  Arxiv 0906.49603.



\bibitem{konno} Kitagawa S., Konno K.:
\emph{Fibred rational surfaces with extremal Mordell-Weill latices.} Mathematische Zeitschrift, 251, 179-204, 2005.





\bibitem{konno3} Konno K.:
\emph{Clifford index and the slope of fibered surfaces.} J. Alg. Geom. 8, 207-220, 1999.


\bibitem{martens} Martens G.:
\emph{The gonality of curves on a Hirzebruch surface.} Arch. Math., Vol 67, 349-352, 1996.


\bibitem{parshin}  Parshin A.:
\emph{Algebraic curves over functions
fields}. Izvest. Akad. Nauk., 32, 1968.

\bibitem{reider} Reider I.:
\emph{Vector Bundles of rank $2$ and linear
systems on algebraic surfaces.} Ann. Math. 127, 309-316, 1988.

 \bibitem{tan} Tan S-L:
\emph{The minimal number of singular fibers of a semi stable curve over $\Bbb P^1$.}  J. Algebraic Geom., 4, 591--596, 1995.


\bibitem{tan-tu-zamora} Tan S-L , Tu Y. ,  Zamora A. G. :
\emph{On complex surfaces with $5$ or $6$ semistable singular fibers over
$\mathbb{P}^1$ }. Math. Zeitschrift, No. 249,  427-438, 2005.

\bibitem{vojta} Vojta P.:
\emph{Diophantine Inequalities and Arakelov Theory }. Appendix to introduction to Arakelov theory by S. Lan , Springer-Verlag, 155-178, 1988.




\end{thebibliography}
\end{document}